\documentclass[11pt]{article}
\usepackage{latexsym}
\usepackage{amsmath}
\usepackage{amsthm,color}
\usepackage{amssymb}
\usepackage{mathrsfs}
\usepackage[pagewise]{lineno}
\usepackage[colorlinks,  linkcolor=blue,  anchorcolor=blue, citecolor=blue]{hyperref}

\usepackage{lscape}
\newtheorem{theorem}{\indent Theorem}[section]

\newtheorem{lemma}[theorem]{\indent Lemma}
\newtheorem{remark}{\indent Remark}[section]

\begin{document}
\renewcommand{\baselinestretch}{1.3}


\begin{center}
    {\large \bf   Positive ground state solutions for fractional Laplacian system with one critical exponent and one subcritical exponent}
\vspace{0.5cm}\\
{\sc Maoding Zhen$^{1,2}$, Jinchun He$^{1,2} $,Haoyuan Xu$^{1,2,\ast} $,Meihua Yang$^{1,2} $}\\
{\small 1) School of Mathematics and Statistics, Huazhong University of Science and Technology,\\ Wuhan 430074, China}\\
{\small 2) Hubei Key Laboratory of Engineering Modeling and Scientific Computing, Huazhong University of Science and Technology,}
{\small Wuhan, 430074, China}
\end{center}


\renewcommand{\theequation}{\arabic{section}.\arabic{equation}}
\numberwithin{equation}{section}


\begin{abstract}
In this paper, we consider the following fractional Laplacian system with one critical exponent and one subcritical exponent
\begin{equation*}
\begin{cases}

(-\Delta)^{s}u+\mu u=|u|^{p-1}u+\lambda v & x\in \ \mathbb{R}^{N},\\

(-\Delta)^{s}v+\nu v = |v|^{2^{\ast}-2}v+\lambda  u&  x\in \ \mathbb{R}^{N},\\
\end{cases}
\end{equation*}
where $(-\Delta)^{s}$ is the fractional Laplacian, $0<s<1,\ N>2s, \ \lambda <\sqrt{\mu\nu },\ 1<p<2^{\ast}-1~ and~\  2^{\ast}=\frac{2N}{N-2s}$~ is the Sobolev critical exponent.
By using the Nehari\ manifold, we show that there exists a $\mu_{0}\in(0,1)$, such that when $0<\mu\leq\mu_{0}$, the system has a positive ground state solution. When $\mu>\mu_{0}$, there exists a $\lambda_{\mu,\nu}\in[\sqrt{(\mu-\mu_{0})\nu},\sqrt{\mu\nu})$ such that if $\lambda>\lambda_{\mu,\nu}$, the system has a positive ground state solution, if $\lambda<\lambda_{\mu,\nu}$, the system has no  ground state solution.

\end{abstract}
\vspace{-1 cm}

\footnote[0]{ \hspace*{-7.4mm}
$^{*}$ Corresponding author.
The authors were supported by the NSFC grant 11571125.\\
}

\section{Introduction}
\hspace{0.6cm} In the past decades, the Laplacian equation or system has been widely investigated and there are many results  about ground state solutions, multiple positive solutions, sign-changing solutions, etc(see \cite{PWW,PPW,CZ1,CZ2,CZ3,CZ6,CM} and references therein).

Compared to the Laplacian problem, the fractional Laplacian problem is non-local and more challenging. Recently, a great attention has been focused on the study of fractional and non-local operators of elliptic type, both for
the pure mathematical research and in view of concrete real-world applications(see \cite{BCPS,CDS2,MF,CRS,SV2,CS2,SZY,GLZ} and references therein). This type of operator arises in a quite natural way in many different contexts, such as, the thin obstacle problem, finance, phase transitions, anomalous diffusion, flame propagation and many others(see\cite{GGP,DPV,MMC,LS} and references therein).

For the case of fractional Laplacian equation, the existence and nonexistence of solutions has been studied by many researchers. For example
\begin{align}\label{ddd}
(-\Delta)^{s}u+u=f(u) \ \ \  \text{in} \ \ \mathbb{R}^{N}
\end{align}
has been studied by many authors under various hypotheses on the nonlinearity $f$. Such as, Wang and Zhou \cite{ZWH} obtained the existence of a radial sign-changing solution for equation \eqref{ddd} by using variational
method and Brouwer degree theory. When the nonlinearity $f$ satisfies the general hypotheses introduced by Berestycki and Lions \cite{HBPL}, Chang and Wang \cite{XCZ} proved the existence of a radially symmetric ground state solution with the help of the Pohoz\v{a}ev identity for \eqref{ddd}. However, in all these works, they only consider the existence and nonexistence solutions, but there are few results about the uniqueness of solution for fractional Laplacian equation. In the remarkable papers \cite{RLFL}\cite{FLS}, for the subcritical case, when $f(u)=|u|^{p-2}u,\ p\in(2,2^{\ast})$, R.L. Frank and E. Lenzmann \cite{RLFL} showed the uniqueness of non-linear ground states solutions to the equation \eqref{ddd} for one dimension case and R.L. Frank, E. Lenzmann and L. Silvestre \cite{FLS} showed the general unique ground state solution  to the equation \eqref{ddd} for dimension great than one.

It is also nature to study the coupled system. For the following fractional Laplacian system,
\begin{equation*}
\begin{cases}

(-\Delta)^{s}u=F(u,v) & x\in \ \ \mathbb{R}^{N},\\

(-\Delta)^{s}v= G(u,v)&  x\in \ \ \mathbb{R}^{N},\\
\end{cases}
\end{equation*}
has been investigated  by many authors under various hypotheses on the nonlinearity $F(u,v)$ and $G(u,v)$. For example,
when $F(u,v)=f(u)+\lambda v-u, G(u,v)=g(u)+\lambda u-v $, D.F. L\"{u}, S.J. Peng \cite{DFP} showed that under suitable condition of $f,g$, it has a vector ground state solution for $\lambda\in(0,1)$. When $F(u,v)=\mu_{1}|u|^{2^{\ast}-2}u+\frac{\alpha\gamma}{2^{\ast}}|u|^{\alpha-2}u|v|^{\beta}, G(u,v)=\mu_{2} |v|^{2^{\ast}-2}v+\frac{\beta\gamma}{2^{\ast}}|u|^{\alpha}|v|^{\beta-2}v$,
M. D. Zhen, J. C. He and H. Y. Xu \cite{ZHX} showed that the existence and nonexistence of ground state solutions under suitable condition of $\alpha,\beta,\gamma,s, N$ and Z. Guo, S. Luo and W. Zou \cite{GLZ} showed that under suitable condition of $\alpha,\beta,s, N$ the system has a positive ground state solution for all $\gamma>0$. When $F(u,v)=(|u|^{2p}+b|u|^{p-1}|v|^{p+1})u-u, G(u,v)=(|v|^{2p}+b|v|^{p-1}|u|^{p+1})v-\omega ^{2\alpha}v$, Q. Guo and X. He \cite{QGX} proved the existence of a least energy solution via Nehari manifold method and showed
that if $b$ is large enough, it has a positive least energy solution.
Note that in all these works, they only consider subcritical case or critical case.  As far as we know, there are few results for the fractional Laplacian system with one subcritical equation and one critical equation.  In this paper, we consider the following fractional Laplacian system with one critical exponent and one subcritical exponent on  ${\mathbb R}^N$. In the case of Laplacian system, the problem has been investigated by Z. Chen, W. Zou in \cite{CZ6}.

The system we consider is the following
 \begin{equation}\label{int1}
\begin{cases}

(-\Delta)^{s}u+\mu u=|u|^{p-1}u+\lambda  v & x\in \ \ \mathbb{R}^{N},\\

(-\Delta)^{s}v+\nu v= |v|^{2^{\ast}-2}v+\lambda  u&  x\in \ \ \mathbb{R}^{N},\\
\end{cases}
\end{equation}
where $(-\Delta)^{s}$ is the fractional Laplacian, $0<s<1,\ N>2s, \ \lambda <\sqrt{\mu\nu },\ 1<p<2^{\ast}-1,\  2^{\ast}=\frac{2N}{N-2s}$ is the Sobolev critical exponent.
The fractional Laplacian\ $(-\Delta)^{s}$\ is defined by
$$
-(-\Delta)^{s}u(x)=\frac{C(N,s)}{2}\int_{\mathbb{R}^{N}}\frac{u(x+y)+u(x-y)-2u(x)}{|y|^{N+2s}}dy,\ \ x\in\mathbb{R}^{N}
$$ with

$$
C(N,s)=\left(\int_{\mathbb{R}^{N}}\frac{1-\cos(\varsigma_{1})}{|\varsigma|^{N+2s}}d\varsigma\right)^{-1}=2^{2s}\pi^{-\frac{N}{2}}\frac{\Gamma(\frac{N+2s}{2})}{\Gamma(2-s)}s(1-s).
$$

Let $D^{s}(\mathbb{R}^{N})$ be the Hilbert space defined as the completion of $C^{\infty}_0(\mathbb{R}^N)$ with the scalar product
$$
\langle u,v\rangle_{D^{s}(\mathbb{R}^{N})}=\frac{C(N,s)}{2}\int_{\mathbb{R}^{N}}\int_{\mathbb{R}^{N}}\frac{(u(x)-u(y))(v(x)-v(y))}{|y-x|^{N+2s}}dxdy,
$$and norm
$$
\|u\|^{2}_{D^{s}(\mathbb{R}^{N})}=\int_{\mathbb{R}^{N}}|(-\Delta)^{\frac{s}{2}}u|^{2}dx=\frac{C(N,s)}{2}\int_{\mathbb{R}^{N}}\int_{\mathbb{R}^{N}}\frac{|u(x)-u(y)|^{2}}{|y-x|^{N+2s}}dxdy.
$$
Let ${H^{s}(\mathbb{R}^{N})}$ be the Hilbert space of function in $\mathbb{R}^N$ endowed with the standard scalar product and norm
$$\langle u,v\rangle=\int_{\mathbb{R}^{N}}\left((-\Delta)^{\frac{s}{2}} u\cdot(-\Delta)^{\frac{s}{2}}v+uv\right)dx,\ \ \|u\|_{H^{s}(\mathbb{R}^{N})}^{2}=\langle u,u\rangle.$$
Let $C_{p+1}$ be the sharp constant of the Sobolev embedding ${H^{s}(\mathbb{R}^{N})}\hookrightarrow L^{p+1}(\mathbb{R}^{N})$
\begin{align}\label{intabc}
C_{p+1}=\inf \limits_{u\in {H^{s}(\mathbb{R}^{N})})\setminus \{0\}}\frac{\|u\|^{2}_{H^{s}(\mathbb{R}^{N})}}{(\int_{\mathbb{R}^{N}}|u|^{p+1}dx)^{\frac{2}{p+1}}},
\end{align}
and let $S_s$ be the sharp imbedding constant of ${D^{s}(\mathbb{R}^{N})}\hookrightarrow L^{2^{\ast}}(\mathbb{R}^{N})$
\begin{align}\label{int2}
S_{s}=\inf \limits_{u\in {D^{s}(\mathbb{R}^{N})})\setminus \{0\}}\frac{\|u\|^{2}_{D^{s}(\mathbb{R}^{N})}}{(\int_{\mathbb{R}^{N}}|u|^{2^{\ast}}dx)^{\frac{2}{2^{\ast}}}}.
\end{align}

From \cite{CT} we have  $S_{s}$\ is attained in\ $\mathbb{R}^{N}$ by \ $\widetilde{u}_{\epsilon, y}=\kappa(\varepsilon^{2}+|x-y|^{2})^{-\frac{N-2s}{2}}$, where $\kappa\neq 0 \in \mathbb{R},\ \varepsilon>0 $  and $y\in
 \mathbb{R}^{N}$.

Denote $\mathcal{H}={H^{s}(\mathbb{R}^{N})}\times {H^{s}(\mathbb{R}^{N})}$ and $\mathcal{D}={D^{s}(\mathbb{R}^{N})}\times {D^{s}(\mathbb{R}^{N})}$, with the norm given by
$$
\|(u,v)\|_{\mathcal{H}}^{2}=\|u\|_{H^{s}(\mathbb{R}^{N})}^{2}+\|v\|_{H^{s}(\mathbb{R}^{N})}^{2}=\|(u,v)\|_{\mathcal{D}}^{2}+\|u\|^{2}_{L^2(\mathbb{R}^N)}+\|v\|^{2}_{L^2(\mathbb{R}^N)},
$$where $||(u,v)||^{2}_{\mathcal{D}}=\|u\|^{2}_{D^{s}(\mathbb{R}^{N})}+\|v\|^{2}_{D^{s}(\mathbb{R}^{N})}$.

The energy functional associated with \eqref{int1} is given by
\begin{align}\label{pp}
E_{\mu,\nu,\lambda}(u,v)=&\frac{1}{2}||(u,v)||^{2}_{\mathcal{D}}+\frac{1}{2}\int_{\mathbb{R}^{N}}(\mu u^{2}+\nu v^{2})dx\\\nonumber
&-\frac{1}{p+1}\int_{\mathbb{R}^{N}}|u|^{p+1}dx-\frac{1}{2^{\ast}}\int_{\mathbb{R}^{N}}|v|^{2^{\ast}}dx-\lambda \int_{\mathbb{R}^{N}}uvdx.
\end{align}

Define the Nehari manifold
\begin{align*}
\mathbb{M}:=\mathbb{M}_{\lambda}:=\mathbb{M}_{\mu,\nu,\lambda}&=\{(u,v)\in \mathcal{H}\backslash \{(0,0)\},||(u,v)||^{2}_{\mathcal{D}}+\int_{\mathbb{R}^{N}}(\mu u^{2}+\nu v^{2})dx\\
&=\int_{\mathbb{R}^{N}}(|u|^{p+1}+|v|^{2^{\ast}})dx+2\lambda \int_{\mathbb{R}^{N}}uvdx\},
\end{align*}

We say that $(u,v)$ is a nontrivial solution of \eqref{int1} if $u\neq0,v\neq0$ and $(u,v)$ solves \eqref{int1}. Any nontrivial solution of \eqref{int1}\ is in $\mathbb{M}$. Due to the fact that if we take  $\varphi,\psi\in\mathcal{C}^{\infty}_{0}(\mathbb{R}^{N})$ with $\varphi,\psi\not\equiv0$ and $supp(\varphi)\bigcap supp(\psi)=\emptyset$, then there exist $t_{1},t_{2}>0$ such that $(t_{1}\varphi,t_{2}\psi)\in\mathbb{M}$, so   $\mathbb{M}\neq\emptyset$.

Let
\begin{align}\label{int43}
\mu_{0}=\left[\frac{2s(p+1)}{N(p-1)}S^{\frac{N}{2s}}_{s}C^{-\frac{p+1}{p-1}}_{p+1}\right]^{(\frac{p+1}{p-1}-\frac{N}{2s})^{-1}}.
\end{align}
Our main result is:
\begin{theorem}\label{Th1}
Assume $N>2s,\ 1<p<2^{\ast}-1$ and $\mu,\nu>0,\ 0<\lambda<\sqrt{\mu\nu}$. Let $\mu_{0}$ be in \eqref{int43}.\\
$(1)$ If $0<\mu\leq\mu_{0},$ then the system \eqref{int1} has a positive ground state solution.\\
$(2)$ If $\mu>\mu_{0},$ then there exists $\lambda_{\mu,\nu}\in[\sqrt{(\mu-\mu_{0})\nu},\sqrt{\mu\nu})$ such that,

$(i)$ if $\lambda<\lambda_{\mu,\nu}$, the system \eqref{int1} has no ground state solution.

$(ii)$ if $\lambda>\lambda_{\mu,\nu}$, the system \eqref{int1} has a positive ground state solution.
 \end{theorem}
 We sketch our idea of the proof. It is well known that the Sobolev embedding ${H^{s}(\mathbb{R}^{N})}\hookrightarrow L^{p}(\mathbb{R}^{N})$  are not compact for $2\leq p\leq2^{\ast}$. Hence, the associated functional of problem \eqref{int1} does not satisfy the Palais-Smale condition. In order to overcome the lack of compactness, we first set our work space in $H_{r}^{s}(\mathbb{R}^{N})\times H_{r}^{s}(\mathbb{R}^{N})$, where
 $${H_{r}^{s}(\mathbb{R}^{N})}=\{\varphi\in H^{s}(\mathbb{R}^{N}): \varphi \ \text{is radial}\}$$
 and $H^{s}_{r}(\mathbb{R}^{N})$ is endowed with the $H^{s}(\mathbb{R}^{N})$ topology: $\|\varphi\|_{H^{s}_{r}(\mathbb{R}^{N})}=\|\varphi\|_{H^{s}(\mathbb{R}^{N})}$.
Let $${D_{r}^{s}(\mathbb{R}^{N})}=\{\varphi\in D^{s}(\mathbb{R}^{N}): \varphi \ \text{is radial}\}.$$

 By  properties of symmetric radial decreasing rearrangement, we know $C_{p+1}$ is achieved by radial functions in ${H^{s}(\mathbb{R}^{N})}$ and $S_{s}$ is achieved by radial functions in ${D^{s}(\mathbb{R}^{N})}$.  By principle of symmetric criticality (Theorem 1.28 in \cite{MWM}), the solutions for \eqref{int1} in function space ${H_{r}^{s}(\mathbb{R}^{N})}\times {H_{r}^{s}(\mathbb{R}^{N})}$ are also the solutions in function space ${H^{s}(\mathbb{R}^{N})}\times {H^{s}(\mathbb{R}^{N})}$.

  Second, we show that if the critical value $c$ of the functional \eqref{pp} is strictly less than $\frac{s}{N}S^{\frac{N}{2s}}_{s}$, then the corresponding critical sequence will satisfy $(PS)_{c}$ condition.

  Finally, we prove that mountain pass value for \eqref{pp} is less than $\frac{s}{N}S^{\frac{N}{2s}}_{s}$ under some proper conditions.
\begin{remark}
The proof of Theorem \ref{Th1} is totally variational. We do not need to use the regularity of the solutions to the system \eqref{int1}(For the fractional Laplacian equation, there are no general results for regularity(of the solutions) higher than the one derived from the Sobolev imbedding). The method we use here is different from the one used in\cite{CZ6}, where they use a limiting argument to deal with the problem in Laplacian case and the $C^{2}$ regularity of the solutions are needed.
\end{remark}

The paper is organized as follows. In section \ref{S2}, we introduce some preliminaries that will be used to prove Theorem \ref{Th1}. In section \ref{sec1}, we prove Theorem \ref{Th1}.

\section{Some Preliminaries }\label{S2}
As mentioned earlier, we will only work in the radial function space. Set $\mathcal{H}_{r}$ and $\mathcal{D}_{r}$  as the following
$$
\mathcal{H}_{r}=\{(u,v)\in \mathcal{H}: u,~v \ \text{are radial}\},
$$

$$
\mathcal{D}_{r}=\{(u,v)\in \mathcal{D}: u,~v \ \text{are radial}\},
$$
with norm deduced from $\mathcal H$ and $\mathcal D$ respectively. Define the Nehari Manifold in $\mathcal{H}_r$ as
$$
\mathbb{M}_{r}=\{(u,v)\in \mathbb{M}: u,~v \ \text{are radial}\}.
$$

Theorem \ref{Th1} is proved by  Mountain Pass Theorem \cite{AAPH}, we first show that $E_{\mu,\nu,\lambda}$ has a $(PS)_{c}$ sequence in $\mathcal{H}_r$. Choose $\varphi,\psi\in\mathcal{C}^{\infty}_{0}(\mathbb{R}^{N})\cap H_r(\mathbb{R}^N)$ with $\varphi,\psi\not\equiv0$ and $supp(\varphi)\bigcap supp(\psi)=\emptyset$, then there exists $t_{0}>0$ such that $E_{\mu,\nu,\lambda}(t_{0}\varphi,t_{0}\psi)<0$ for all $t\ge t_0$. Take $(u_{0},v_{0})=(t\varphi,t\psi)$ with $t\ge t_0$ large enough. Let
$$
\Gamma=\{\gamma\in\mathcal{C}([0,1],\mathcal{H}_{r}):\gamma(0)=(0,0),\quad \gamma(1)=(u_{0},v_{0})\}.
$$
Define
$$
\widehat{A}_{\mu,\nu,\lambda}:=\inf \limits_{\gamma\in \Gamma}\max \limits_{t\in [0,1]}E_{\mu,\nu,\lambda}(\gamma(t)).
$$

\begin{lemma}\label{Lem1}
Under the condition $0<\lambda<\sqrt{\mu\nu}$, there exists a Palais-Smale sequence $\{(u_{n},v_{n})\}\subset \mathcal{H}_{r}$ such that
\begin{align}\label{int4}
E_{\mu,\nu,\lambda}(u_{n},v_{n})\rightarrow\widehat{A}_{\mu,\nu,\lambda}\ \text{and} \ E'_{\mu,\nu,\lambda}(u_{n},v_{n})\rightarrow0 \ \text{as} \ n\rightarrow+\infty.
\end{align}
\end{lemma}
\begin{proof}
We first claim that $E_{\mu,\nu,\lambda}$ possesses a mountain pass geometry around $(0,0)$;
\newline
(1) there exist $\alpha,\rho>0$, such that $E_{\mu,\nu,\lambda}(u,v)>\alpha$ for all $||(u,v)||_{\mathcal{H}_{r}}=\rho$;
\newline
(2) there exist  $(u_{0},v_{0})\in\mathcal{H}_{r}$ such that $||(u_{0},v_{0})||_{\mathcal{H}_{r}}>\rho$ and $E_{\mu,\nu,\lambda}(u_{0},v_{0})<0$.

To claim (1), since $\lambda<\sqrt{\mu\nu }$, we can take a small $\tau>0$ such that $\lambda^2=(\mu-\tau)(\nu-\tau)$, then by Sobolev imbedding,
\begin{align*}
E_{\mu,\nu,\lambda}(u,v)&=\frac{1}{2}||(u,v)||^{2}_{\mathcal{D}_{r}}+\frac{1}{2}\int_{\mathbb{R}^{N}}(\mu u^{2}+\nu v^{2})dx-\frac{1}{p+1}\int_{\mathbb{R}^{N}}|u|^{p+1}dx\\
&-\frac{1}{2^{\ast}}\int_{\mathbb{R}^{N}}|v|^{2^{\ast}}dx-\lambda \int_{\mathbb{R}^{N}}uvdx\\
&\geq\frac{1}{2}||(u,v)||^{2}_{\mathcal{D}_{r}}+\frac{\tau}{2}\int_{\mathbb{R}^N}(u^2+v^2)-\frac{1}{p+1}\int_{\mathbb{R}^{N}}|u|^{p+1}dx-\frac{1}{2^{\ast}}\int_{\mathbb{R}^{N}}|v|^{2^{\ast}}dx\\
&\geq\min(\frac{1}{2},\frac{\tau}{2})||(u,v)||^{2}_{\mathcal{H}_{r}}-C_{1}||(u,v)||^{p+1}_{\mathcal{H}_{r}}-C_{2}||(u,v)||^{2^{\ast}}_{\mathcal{H}_{r}}.
\end{align*}

Choose $\rho>0$ sufficiently small, if $||(u,v)||^{2}_{\mathcal{H}_{r}}=\rho$, then
$$
E_{\mu,\nu,\lambda}(u,v)\geq \min(\frac{1}{2},\frac{\tau}{2}) ||(u,v)||^{2}_{\mathcal{H}_{r}}-C_{1}||(u,v)||^{p+1}_{\mathcal{H}_{r}}-C_{2}||(u,v)||^{2^{\ast}}_{\mathcal{H}_{r}}>\frac{1}{4}C_{3}\rho^{2}>0.
$$

(2) is obvious when we choose $t$ large enough such that $\|(u_0,v_0)\|_{\mathcal H_{r}}=t\|(\varphi,\psi)\|_{\mathcal{H}_{r}}>\rho$.

By Mountain Pass Theorem \cite{AAPH}, for the constant $0<\widehat{A}_{\mu,\nu,\lambda}:=\inf \limits_{\gamma\in \Gamma}\max \limits_{t\in [0,1]}E_{\mu,\nu,\lambda}(\gamma(t)),$ there exist a $(PS)_{\widehat{A}_{\mu,\nu,\lambda}}$ sequence $\{(u_{n},v_{n})\}\subset \mathcal{H}_{r}$, that is
\begin{align*}
E_{\mu,\nu,\lambda}(u_{n},v_{n})\rightarrow\widehat{A}_{\mu,\nu,\lambda}\ \text{and} \ E'_{\mu,\nu,\lambda}(u_{n},v_{n})\rightarrow0 \ \text{as} \ n\rightarrow+\infty,
\end{align*}
where $$\Gamma=\{\gamma\in\mathcal{C}([0,1],\mathcal{H}_{r}):\gamma(0)=(0,0),\gamma(1)=(u_{0},v_{0})\}.$$
\end{proof}

Define
$$
\overline{A}_{\mu,\nu,\lambda}=\inf \limits_{\mathcal{H}_{r}\setminus\{(0,0)\}}\max \limits_{t>0}E_{\mu,\nu,\lambda}(tu,tv),
$$
and
\begin{align*}
A_{\mu,\nu,\lambda}:=\inf \limits_{(u,v)\in \mathbb{M}_r}E_{\mu,\nu,\lambda}(u,v)=\inf \limits_{(u,v)\in \mathbb{M}_r}\left((\frac{1}{2}-\frac{1}{p+1})\int_{\mathbb{R}^{N}}(|u|^{p+1}dx+\frac{s}{N}\int_{\mathbb{R}^{N}}|v|^{2^{\ast}}dx\right).\\
\end{align*}
\begin{lemma}\label{Lem5}
Under the condition $0<\lambda<\sqrt{\mu\nu}$, $\widehat{A}_{\mu,\nu,\lambda}= \overline{A}_{\mu,\nu,\lambda}=A_{\mu,\nu,\lambda}$.
\end{lemma}
\begin{proof}
We first claim $\overline{A}_{\mu,\nu,\lambda}=A_{\mu,\nu,\lambda}$. For any $(u,v)\in\mathcal{H}_{r}$ with $(u,v)\not\equiv(0,0)$, there exist a unique $t_{\lambda,u,v}>0$ such that
\begin{align}\label{Lem2}
\max \limits_{t>0}E_{\mu,\nu,\lambda}(tu,tv)&=E_{\mu,\nu,\lambda}(t_{\lambda,u,v}u,t_{\lambda,u,v}v)\\\nonumber
&=(\frac{1}{2}-\frac{1}{p+1})t^{p+1}_{\lambda,u,v}\int_{\mathbb{R}^{N}}|u|^{p+1}dx+t^{2^{\ast}}_{\lambda,u,v}\left(\frac{1}{2}-\frac{1}{2^{\ast}}\right)\int_{\mathbb{R}^{N}}|v|^{2^{\ast}}dx,
\end{align}
where $t_{\lambda,u,v}>0$ satisfies $\varphi(\lambda,u,v,t_{\lambda,u,v})=0$ and
\begin{align*}
\varphi(\lambda,u,v,t)&=||(u,v)||^{2}_{\mathcal{D}_{r}}+\int_{\mathbb{R}^{N}}(\mu u^{2}+\nu v^{2})dx\\
&-t^{p-1}\int_{\mathbb{R}^{N}}|u|^{p+1}dx-t^{2^{\ast}-2}\int_{\mathbb{R}^{N}}|v|^{2^{\ast}}dx-2\lambda \int_{\mathbb{R}^{N}}uvdx.
\end{align*}

Which implies that $(t_{\lambda,u,v}u,t_{\lambda,u,v}v)\in \mathbb{M}_{r}$.  From \eqref{Lem2}, we get for any $(u,v)\in \mathcal H_r$, $(u,v)\not\equiv (0,0)$,
$$
\max \limits_{t>0}E_{\mu,\nu,\lambda}(tu,tv)=E_{\mu,\nu,\lambda}(t_{\lambda,u,v}u,t_{\lambda,u,v}v)\ge A_{\mu,\nu,\lambda},
$$
therefore $\overline{A}_{\mu,\nu,\lambda}\ge A_{\mu,\nu,\lambda}$. Similarly, we can show that $\overline{A}_{\mu,\nu,\lambda}\le A_{\mu,\nu,\lambda}$, thus $\overline{A}_{\mu,\nu,\lambda}= A_{\mu,\nu,\lambda}$.

Next, we claim $\widehat{A}_{\mu,\nu,\lambda}= \overline{A}_{\mu,\nu,\lambda}.$

For any $\epsilon>0$, we can take a $(u,v)\not\equiv(0,0)$ such that
\begin{equation}
E_{\mu,\nu,\lambda}(t_{\lambda,u,v}u,t_{\lambda,u,v}v)=\max \limits_{t>0}E_{\mu,\nu,\lambda}(tu,tv)<\overline{A}_{\mu,\nu,\lambda}+\epsilon.
\end{equation}
We take a two dimensional space $\mathcal{S}$ in $\mathcal{H}_{r}$ contain $(u,v)$ and $(u_{0},v_{0})$ and we choose a large $R>\max(\|(t_{\lambda,u,v}u,t_{\lambda,u,v}v)\|,\|(u_{0},v_{0})\|)>0$, such that $E_{\mu,\nu,\lambda}(u,v)<0$ for all $(u,v)\in \mathcal{S}$ with $\|(u,v)\|=R$.
Now, we define a path $\Gamma$ connecting $(0,0)$ and $(u_{0},v_{0})$ as follows. If $t\in [0,\frac{1}{2}]$, let $(u_{t},v_{t})=(\frac{2tRu}{\|u\|},\frac{2tRv}{\|v\|})$ which is the segment connecting $(0,0)$ and $(\frac{Ru}{\|u\|},\frac{Rv}{\|v\|})$. If $t\in(\frac{1}{2},\frac{3}{4})$, let $(u_{t},v_{t})$ be the arc in $\mathcal{S}$, satisfying $(u_{\frac{1}{2}},v_{\frac{1}{2}})=(\frac{Ru}{\|u\|},\frac{Rv}{\|v\|}),\ \ (u_{\frac{3}{4}},v_{\frac{3}{4}})=(\frac{Ru_{0}}{\|u_{0}\|},\frac{Rv_{0}}{\|v_{0}\|})$ and $\|(u_{t},v_{t})\|=R.$ If $t\in(\frac{3}{4},1]$ define $(u_{t},v_{t})=(\frac{4(1-t)Ru_{0}}{\|u_{0}\|}+(4t-3)u_{0},\frac{4(1-t)Rv_{0}}{\|v_{0}\|}+(4t-3)v_{0})$ which is a segment connecting $(\frac{Ru_{0}}{\|u_{0}\|},\frac{Rv_{0}}{\|v_{0}\|})$ and $(u_{0},v_{0}).$ Let $\gamma(t)=(u_{t}(\cdot),v_{t}(\cdot)), t\in[0,1],$ then it is easy to check that $$\overline{A}_{\mu,\nu,\lambda}+\epsilon>\max \limits_{k>0}E_{\mu,\nu,\lambda}(ku,kv)=\max \limits_{t\in[0,\frac{1}{2}]}E_{\mu,\nu,\lambda}(u_{t},v_{t})=\max \limits_{t\in[0,1]}E_{\mu,\nu,\lambda}(\gamma(t))\geq \widehat{A}_{\mu,\nu,\lambda},$$
which implies that $\overline{A}_{\mu,\nu,\lambda}\geq \widehat{A}_{\mu,\nu,\lambda}.$ Thus $\widehat{A}_{\mu,\nu,\lambda}\leq \overline{A}_{\mu,\nu,\lambda}=A_{\mu,\nu,\lambda}.$

Let $\{(u_{n},v_{n})\}$ is a $(PS)_{\widehat{A}_{\mu,\nu,\lambda}}$ sequence, then
$$
E_{\mu,\nu,\lambda}(u_{n},v_{n})\rightarrow\widehat{A}_{\mu,\nu,\lambda}\ \text{and}\  (E_{\mu,\nu,\lambda})'(u_{n},v_{n})\rightarrow0,\ \text{as}\ n\rightarrow+\infty.
$$

We claim $\{(u_{n},v_{n})\}$ is bounded in $\mathcal{H}_{r}$.
For $n$ large enough, we have since $\lambda<\sqrt{\mu\nu }$, we can take a small $\tau>0$ such that $\lambda^2=(\mu-\tau)(\nu-\tau)$, then by Sobolev imbedding
\begin{align}\label{q}
\widehat{A}_{\mu,\nu,\lambda}+o(1)(\|(u_{n},v_{n})\|)&=E_{\mu,\nu,\lambda}(u_{n},v_{n})-\frac{1}{p+1}<E'_{\mu,\nu,\lambda}(u_{n},v_{n}),(u_{n},v_{n})>\\\nonumber
&=(\frac{1}{2}-\frac{1}{p+1})\left[||(u,v)||^{2}_{\mathcal{D}_{r}}+\int_{\mathbb{R}^{N}}(\mu u^{2}+\nu v^{2})dx-2\lambda \int_{\mathbb{R}^{N}}uvdx\right]\\\nonumber
&+(\frac{1}{p+1}-\frac{1}{2^{\ast}})\int_{\mathbb{R}^{N}}u^{p+1}_{n}dx\\\nonumber
&\geq (\frac{1}{2}-\frac{1}{p+1})\left[||(u,v)||^{2}_{\mathcal{D}_{r}}+\tau\int_{\mathbb{R}^{N}}( u^{2}+ v^{2})dx\right]\\\nonumber
&\geq\min(1,\tau)(\frac{1}{2}-\frac{1}{p+1})||(u,v)||^{2}_{\mathcal{H}_{r}}.
\end{align}
Consequently, $\{(u_{n},v_{n})\}$ is bounded in $\mathcal{H}_{r}$.  Since $\widehat{A}_{\mu,\nu,\lambda}>0$,  $\{(u_{n},v_{n})\}$ is not $(0,0)$.

Let $(t_nu_{n},t_nv_{n})\in \mathbb M_r$, from
$$
\langle (E_{\mu,\nu,\lambda})'(u_{n},v_{n}),(u_{n},v_{n})\rangle\to 0\ \ \text{as}\ \ n\rightarrow+\infty,
$$
it is easy to see that $t_n\to 1$ as $n\to \infty$.
So
\begin{align*}
\widehat{A}_{\mu,\nu,\lambda}&=\lim \limits_{n\rightarrow\infty}E_{\mu,\nu,\lambda}(u_{n},v_{n})=\lim \limits_{n\rightarrow\infty}E_{\mu,\nu,\lambda}(t_nu_{n},t_nv_{n})\\
&\ge \liminf\limits_{n\to\infty}E_{\mu,\nu,\lambda}(t_nu_{n},t_nv_{n})\ge A_{\mu,\nu,\lambda}=\bar{A}_{\mu,\nu,\lambda}.
\end{align*}
Consequently, $\widehat{A}_{\mu,\nu,\lambda}=\overline{A}_{\mu,\nu,\lambda}=A_{\mu,\nu,\lambda}$. This completes the proof of Lemma \ref{Lem5}.
\end{proof}
\begin{remark}
By properties of symmetric radial decreasing rearrangement, it is easy to show that the value $\widehat{A}_{\mu,\nu,\lambda}=\overline{A}_{\mu,\nu,\lambda}=A_{\mu,\nu,\lambda}$ defined in $\mathcal{H}_{r}$ is the same as the value defined in $\mathcal{H}$. So the ground state solution in $\mathcal H_r$ is also the ground state solution in $\mathcal H$.
\end{remark}
In order to prove Theorem \ref{Th1}, we need the following lemma.
\begin{lemma}(\cite{HBEL}{Brezis-Lieb Lemma})\label{k}
Let $\{u_{n}\}\subset L^{p+1}(\mathbb{R}^{N})$, $0<p<\infty$. If $\{u_{n}\}$ is bounded in $L^{p+1}(\mathbb{R}^{N})$ and $u_{n}\rightarrow u\ \text{a.e on} \ \mathbb{R}^{N} $, then $$ \int_{\mathbb{R}^{N}}|u_{n}|^{p+1}dx=\int_{\mathbb{R}^{N}}|u|^{p+1}dx+\int_{\mathbb{R}^{N}}|u_{n}-u|^{p+1}dx+o(1).$$
\end{lemma}

Next, we borrow some ideas form \cite{CZ6}. Define
\begin{align}\label{int40}
f_{\beta,\gamma}(u):=\frac{1}{2}||u||^{2}_{D_{r}^{s}(\mathbb{R}^{N})}+\frac{1}{2}\int_{\mathbb{R}^{N}}\beta |u|^{2}dx-\frac{\gamma}{p+1}\int_{\mathbb{R}^{N}}|u|^{p+1}dx,
\end{align}
\begin{align}\label{int41}
g(v)=\frac{1}{2}||v||^{2}_{D_{r}^{s}(\mathbb{R}^{N})}-\frac{1}{2^{\ast}}\int_{\mathbb{R}^{N}}|v|^{2^{\ast}}dx,
\end{align}
and denote $f_{\beta}=f_{\beta,1}$. Then we have following Lemma.
\begin{lemma}\label{Lem6}
For any $(u,v)\in\mathcal{H}_{r}$ with $u\neq0$ and $v\neq0$, there holds
$$\max \limits_{t>0}E_{\mu,\nu,\lambda}(tu,tv)>\min\{\max \limits_{t>0}f_{\mu-\frac{\lambda^{2}}{\nu}}(tu),\  \max \limits_{t>0}g(tv)\}.$$
\end{lemma}
\begin{proof}
Since $2\lambda uv\leq\frac{\lambda^{2}}{\nu}u^{2}+\nu v^{2}$, we have $$E_{\mu,\nu,\lambda}(tu,tv)\geq f_{\mu-\frac{\lambda^{2}}{\nu}}(tu)+g(tv).$$
Moreover, there exist $t_{1},t_{2}>0$ such that $$\max \limits_{t>0}f_{\mu-\frac{\lambda^{2}}{\nu}}(tu)=f_{\mu-\frac{\lambda^{2}}{\nu}}(t_{1}u),\ \ \max \limits_{t>0}g(tv)=g(t_{2}v).$$
Since $f_{\mu-\frac{\lambda^{2}}{\nu}}(0)=0,\  g(0)=0$ and $f_{\mu-\frac{\lambda^{2}}{\nu}}(tu)$ is increasing in $[0,t_{1}]$, decreasing in $[t_{1},+\infty)$, $g(tv)$ is increasing in $[0,t_{2}]$, decreasing in $[t_{2},+\infty)$. Thus, if $t_{1}<t_{2}$, then $g(t_{1}v)>0$ and $E_{\mu,\nu,\lambda}(t_{1}u,t_{1}v)>f_{\mu-\frac{\lambda^{2}}{\nu}}(t_{1}u)=\max \limits_{t>0}f_{\mu-\frac{\lambda^{2}}{\nu}}(tu).$ If $t_{1}\geq t_{2}$, then $f_{\mu-\frac{\lambda^{2}}{\nu}}(t_{2}u)>0$ and $E_{\mu,\nu,\lambda}(t_{2}u,t_{2}v)>g(t_{2}v)=\max \limits_{t>0}g(tv).$ This completes the proof.
\end{proof}

By \cite{RLFL,FLS}, there exists a unique positive radial ground state solution $w$ to the equation $(-\Delta)^{s}u+u=u^{p}$ in $ H^{s}(\mathbb{R}^N)$. By \eqref{intabc},

$$
f_{1}(w)=\left(\frac{1}{2}-\frac{1}{p+1}\right)C^{\frac{p+1}{p-1}}_{p+1},
$$ where $f_{1}$ is defined in \eqref{int40}. Let $w_{\beta,\gamma}(x):=\beta^{\frac{1}{p-1}}\gamma^{-\frac{1}{p-1}}w(\beta^{\frac{1}{2s}}x)$, then $w_{\beta,\gamma}(x)$ is the unique positive radial solution of $(-\Delta)^{s}u+\beta u=\gamma u^{p},\ u\in H^{s}(\mathbb{R}^{N})$ with the energy
\begin{align}\label{int46}
f_{\beta,\gamma}(w_{\beta,\gamma})=\gamma^{-\frac{2}{p-1}}\beta^{\frac{p+1}{p-1}-\frac{N}{2s}}f_{1}(w)=\left(\frac{1}{2}-\frac{1}{p+1}\right)\gamma^{-\frac{2}{p-1}}\beta^{\frac{p+1}{p-1}-\frac{N}{2s}}C^{\frac{p+1}{p-1}}_{p+1}.
\end{align}
For convenience, we denote $w_{\beta}=w_{\beta,1}$. Define $\alpha=N(\frac{1}{p+1}-\frac{1}{2^{\ast}})\in(0,s),$ then $$\frac{s}{p+1}=\frac{\alpha}{2}+\frac{s-\alpha}{2^{\ast}}.$$
Let $\bar\mu_0=\frac{\alpha}{s}(\frac{s-\alpha}{s})^{\frac{N-2s}{2s}(\frac{p+1}{p-1}-\frac{N}{2s})^{-1}}$,  we have the following Lemma:
\begin{lemma}\label{int50}
There exists a constant $\mu_0>0$ with $0<\mu_{0}<\overline{\mu}_{0}$ such that
\begin{equation}\label{int45}
f_{\mu}(w_{\mu})=\left\{\begin{array}{ll}
                >\frac{s}{N}S^{\frac{N}{2s}}_{s},& \text{if} \ \mu>\mu_{0},\\
                =\frac{s}{N}S^{\frac{N}{2s}}_{s},& \text{if} \ \mu=\mu_{0},\\
                <\frac{s}{N}S^{\frac{N}{2s}}_{s},& \text{if} \ \mu<\mu_{0}.
           \end{array}\right.
\end{equation}
\end{lemma}
\begin{proof}
By \eqref{int43} and \eqref{int46}, we have $f_{\mu_{0}}(w_{\mu_{0}})=\frac{s}{N}S^{\frac{N}{2s}}_{s}.$ Since $p<2^{\ast}-1$, we have $$\frac{p+1}{p-1}-\frac{N}{2s}>0.$$  From \eqref{int46} it is easy to obtain \eqref{int45}.
In order to prove $\overline{\mu}_{0}>\mu_{0}$, by \eqref{int45} we need to show $f_{\overline{\mu}_{0}}(w_{\overline{\mu}_{0}})>\frac{s}{N}S^{\frac{N}{2s}}_{s}.$ By H\"{o}lder inequality and Young inequality, we have
\begin{align*}
\left(\int_{\mathbb{R}^{N}}|u|^{p+1}dx\right)^{\frac{2}{p+1}}&\leq\left(\int_{\mathbb{R}^{N}}|u|^{2}dx\right)^{\frac{\alpha}{s}}\left(\int_{\mathbb{R}^{N}}|u|^{2^{\ast}}dx\right)^{\frac{2(s-\alpha)}{2^{\ast}s}}\\
&\leq\frac{\alpha}{s}\epsilon^{\frac{s}{\alpha}}\int_{\mathbb{R}^{N}}|u|^{2}dx+\frac{s-\alpha}{s}\epsilon^{-\frac{s}{s-\alpha}}\left(\int_{\mathbb{R}^{N}}|u|^{2^{\ast}}dx\right)^{\frac{2}{2^{\ast}}}.
\end{align*}
If we choose $C_{0}>0,\ \ \epsilon_{0}>0$ such that $$C_{0}\frac{\alpha}{s}\epsilon_{0}^{\frac{s}{\alpha}}=1,\ \ C_{0}\frac{s-\alpha}{s}\epsilon_{0}^{-\frac{s}{s-\alpha}}=S_{s},$$ then, we have $$C_{0}=S^{\frac{s-\alpha}{s}}_{s}\left(\frac{s}{s-\alpha}\right)^{\frac{s-\alpha}{s}}\left(\frac{s}{\alpha}\right)^{\frac{\alpha}{s}}.$$ and $$||u||^{2}_{H_{r}^{s}(\mathbb{R}^{N})}+\int_{\mathbb{R}^{N}}|u|^{2}dx>S_{s}\left(\int_{\mathbb{R}^{N}}|u|^{2^{\ast}}dx\right)^{\frac{2}{2^{\ast}}}+\int_{\mathbb{R}^{N}}|u|^{2}dx\geq C_{0}\left(\int_{\mathbb{R}^{N}}|u|^{p+1}dx\right)^{\frac{2}{p+1}}.$$
This implies that $C_{p+1}>C_{0}$. Combining this with \eqref{int46}, we obtain
\begin{align*}
f_{\mu}(w_{\mu})&>\left(\frac{1}{2}-\frac{1}{p+1}\right)\mu^{\frac{p+1}{p-1}-\frac{N}{2s}}\left[S^{\frac{s-\alpha}{s}}_{s}\left(\frac{s}{s-\alpha}\right)^{\frac{s-\alpha}{s}}\left(\frac{s}{\alpha}\right)^{\frac{\alpha}{s}}\right]^{\frac{p+1}{p-1}}\\
&=\frac{s}{N}S^{\frac{N}{2s}}_{s}\left(\frac{s}{s-\alpha}\right)^{\frac{N-2s}{2s}}\left(\frac{s}{\alpha}\right)^{\frac{p+1}{p-1}-\frac{N}{2s}}\mu^{\frac{p+1}{p-1}-\frac{N}{2s}}.
\end{align*}
If we choose $$\overline{\mu}_{0}=\frac{\alpha}{s}(\frac{s-\alpha}{s})^{\frac{N-2s}{2s}(\frac{p+1}{p-1}-\frac{N}{2s})^{-1}},$$
we have $f_{\overline{\mu}_{0}}(w_{\overline{\mu}_{0}})>\frac{s}{N}S^{\frac{N}{2s}}_{s}$. Thus, $\overline{\mu}_{0}>\mu_{0}.$ This completes the proof.
\end{proof}
\begin{remark}
From the proof of the above lemma, though the exact values of $C_{p+1}$ and $\mu_0$ are unknown, we have a lower bound estimate for $C_{p+1}$ and an upper bound estimate for $\mu_0$:
\begin{align*}
C_{p+1}>C_{0}=S^{\frac{s-\alpha}{s}}_{s}\left(\frac{s}{s-\alpha}\right)^{\frac{s-\alpha}{s}}\left(\frac{s}{\alpha}\right)^{\frac{\alpha}{s}},
\end{align*}

\begin{align*}
\mu_{0}<\overline{\mu}_{0}=\frac{\alpha}{s}(\frac{s-\alpha}{s})^{\frac{N-2s}{2s}(\frac{p+1}{p-1}-\frac{N}{2s})^{-1}}.
\end{align*}
Since $\alpha\in(0,s)$, it is easy to deduce that $\overline{\mu}_{0}<1$.
\end{remark}
For any $\mu>\mu_{0},\ \ \nu>0$, we define a $C^{1}$ function $h_{\mu,\nu}:(0,+\infty)\rightarrow \mathbb{R} $ by
\begin{align}\label{int47}
h_{\mu,\nu}(a)=\frac{\mu+\nu a^{2}}{2a}-\frac{\mu_{0}}{2a}\left(1+a^{2}\right)^{-\frac{N}{2s}(\frac{p+1}{p-1}-\frac{N}{2s})^{-1}}.
\end{align}
Then, $$h_{\mu,\nu}(a)>\frac{\mu-\mu_{0}+\nu a^{2}}{2a}\geq\sqrt{(\mu-\mu_{0})\nu}.$$
Thus, $h_{\mu,\nu}(a)\rightarrow+\infty$ as $a\rightarrow 0_{+}$ and $h_{\mu,\nu}(a)$ is increasing in $[\sqrt{\frac{\mu}{\nu}},+\infty).$
Therefore, there exists $a_{\mu,\nu}\in(0,\sqrt{\frac{\mu}{\nu}})$ such that
\begin{align}\label{int48}
\widetilde{\lambda}_{\mu,\nu}:=h_{\mu,\nu}(a_{\mu,\nu})=\min \limits_{a\in(0,+\infty)}h_{\mu,\nu}(a).
\end{align}
Since $h_{\mu,\nu}(\sqrt{\frac{\mu}{\nu}})<\sqrt{\mu\nu}$, we have
\begin{align}\label{int49}
\sqrt{(\mu-\mu_{0})\nu}<\widetilde{\lambda}_{\mu,\nu}<\sqrt{\mu\nu}.
\end{align}

\begin{lemma}\label{Lem7}
(1)If $0<\mu\leq\mu_{0},$ then $A_{\mu,\nu,\lambda}<\frac{s}{N}S^{\frac{N}{2s}}_{s}.$\\
(2)If $\mu>\mu_{0}$, then there exists a $\lambda_{\mu,\nu}\in[\sqrt{(\mu-\mu_{0})\nu},\widetilde{\lambda}_{\mu,\nu})$ such that

$(i)$ if $0<\lambda\leq\lambda_{\mu,\nu},$ then $A_{\mu,\nu,\lambda}=\frac{s}{N}S^{\frac{N}{2s}}_{s}$;

$(ii)$ if $\lambda_{\mu,\nu}<\lambda<\sqrt{\mu\nu},$ then $A_{\mu,\nu,\lambda}<\frac{s}{N}S^{\frac{N}{2s}}_{s},$\\
where $\widetilde{\lambda}_{\mu,\nu}$ is from \eqref{int48}.
\end{lemma}
\begin{proof}
(1) If $\mu\in(0,\mu_{0}),$ by Lemma \ref{int50}, we have $$\max \limits_{t>0}E_{\mu,\nu,\lambda}(tw_{\mu},0)=\max \limits_{t>0}f_{\mu}(tw_{\mu})=f_{\mu}(w_{\mu})<\frac{s}{N}S^{\frac{N}{2s}}_{s}.$$
Thus $A_{\mu,\nu,\lambda}<\frac{s}{N}S^{\frac{N}{2s}}_{s}.$

When $\mu=\mu_{0},$ then $A_{\mu_{0},\nu,\lambda}\leq f_{\mu_{0}}(w_{\mu_{0}})=\frac{s}{N}S^{\frac{N}{2s}}_{s}.$ Assume by contradiction that $A_{\mu_{0},\nu,\lambda}=\frac{s}{N}S^{\frac{N}{2s}}_{s}$, then
$$
E_{\mu_{0},\nu,\lambda}(w_{\mu_{0}},0)=A_{\mu_{0},\nu,\lambda},\ \ (w_{\mu_{0}},0)\in\mathbb{M}_r.
$$
Thus, $(w_{\mu_{0}},0)$ is a ground state solution of \eqref{int1}. Since $\lambda>0$, if $(w_{\mu_{0}},0)$ is a ground state solution of \eqref{int1}, then $w_{\mu_{0}}\equiv0.$ This contract with $w_{\mu_{0}}\neq0,$ so $A_{\mu_{0},\nu,\lambda}<\frac{s}{N}S^{\frac{N}{2s}}_{s}.$

(2) In order to prove second part of Lemma \ref{Lem7}, we divide into two steps.\\
{\bf Step 1} We first show that for any fixed $\mu>\mu_{0},\ \nu>0,$ if $0<\lambda\leq\sqrt{(\mu-\mu_{0})\nu}$, then
\begin{align}\label{int51}
A_{\mu,\nu,\lambda}=\frac{s}{N}S^{\frac{N}{2s}}_{s}.
\end{align}
First, we show that the single equation
\begin{equation}\label{intaa}
(-\Delta)^{s}v+\nu v= |v|^{2^{\ast}-2}v,\ \  v\in H^{s}(\mathbb{R}^{N}),
\end{equation}
 has no nontrivial solutions.

In deed, by the similar arguments as Lemma \ref{Lem1} and Lemma \ref{Lem5}, we know that the energy functional  $I(v)$  for equation \eqref{intaa}  has mountain pass structure and bounded $(PS)_{c}$ sequence. It is well know that if the least energy for $I(v)$ strictly less than $\frac{s}{N}S^{\frac{N}{2s}}_{s}$, then \eqref{intaa} has nontrivial solutions.

Let us now consider the cut-off function $\eta(x)\in C^{\infty}_{0}(\mathbb{R}^{N},[0,1])$ such that $0\leq\eta\leq1,\eta=1\ \text{on}\ B(0,r)$ and $\eta=1\text{on}\ \mathbb{R}^{N}\setminus B(0,2r).$ For every $\epsilon>0,$ we let $v_{\epsilon}=\eta(x)V_{\epsilon}(x)$, where $V_{\epsilon}(x)=\epsilon^{-\frac{N-2s}{2}}V(\frac{x}{\epsilon})$ , $S_{s} $ is attained by $V(x)$. Then, the following estimates holds true(Proposition 21 in \cite{SVa})
 $$\int_{\mathbb{R}^{2N}}\frac{|v_{\epsilon}(x)-v_{\epsilon}(y)|^{2}}{|x-y|^{N+2s}}dxdy\leq S^{\frac{N}{2s}}_{s}+o(\epsilon^{N-2s}),$$
 $$ \int_{\mathbb{R}^{N}}|v_{\epsilon}|^{^{2}}dx=\left\{
\begin{array}{rcl}
C\epsilon^{2s}+o(\epsilon^{N-2s})       &      & if\  N>4s\\
C\epsilon^{2s}\log(\frac{1}{\epsilon})+o(\epsilon^{2s})    &      & if\  N=4s\\
 C\epsilon^{N-2s}+o(\epsilon^{2s})    &      & if\  N<4s
\end{array} \right. $$
 $$\int_{\mathbb{R}^{N}}|v_{\epsilon}|^{^{2^{\ast}}}dx=S^{\frac{N}{2s}}_{s}+o(\epsilon^{N}).$$
 Then, using the above estimates and by the similar arguments as \cite{SVa}, we can show that the least energy for equation \eqref{intaa} can not strictly less than $\frac{s}{N}S^{\frac{N}{2s}}_{s}$. So, \eqref{intaa} has no nontrivial solutions.

 Thus, it is easily seen that $S_{s}$ is also the sharp constant of $$\|v\|^{2}_{D_{r}^{s}(\mathbb{R}^{N})}+\int_{\mathbb{R}^{N}}\nu|v|^{2}dx\geq S_{s}(\int_{\mathbb{R}^{N}}|v|^{2^{\ast}}dx)^{\frac{2}{2^{\ast}}},$$
which implies that
\begin{align}\label{int53}
A_{\mu,\nu,\lambda}\leq\inf \limits_{H_{r}^{s}(\mathbb{R}^{N})\setminus\{0\}}\max \limits_{t>0}E_{\mu,\nu,\lambda}(0,tv)=\frac{s}{N}S^{\frac{N}{2s}}_{s}.
\end{align}
On the other hand, assume $0<\lambda\leq\sqrt{(\mu-\mu_{0})\nu},$ then $\mu-\frac{\lambda^{2}}{\nu}\geq\mu_{0}$. By the same arguments as Lemma \ref{Lem5}, we have $$f_{\mu}(w_{\mu})=\inf \limits_{H_{r}^{s}(\mathbb{R}^{N})\setminus\{0\}}\max \limits_{t>0}f_{\mu}(tu).$$
By \eqref{int2}, we have $$\inf \limits_{H_{r}^{s}(\mathbb{R}^{N})\setminus\{0\}}\max \limits_{t>0}g(tu)=\frac{s}{N}S^{\frac{N}{2s}}_{s}.$$
For any $(u,v)\in \mathcal{H}_{r}\setminus\{(0,0)\},$ if $v=0,$ then $\max \limits_{t>0}E_{\mu,\nu,\lambda}(tu,0)=\max \limits_{t>0}f_{\mu}(tu)\geq\frac{s}{N}S^{\frac{N}{2s}}_{s}.$ If $u=0,$ then $\max \limits_{t>0}E_{\mu,\nu,\lambda}(0,tv)=\max \limits_{t>0}g(tv)\geq\frac{s}{N}S^{\frac{N}{2s}}_{s}.$ If $u\neq0$ and $v\neq0$, then by Lemma \ref{Lem6} and Lemma \ref{int50}, we have
\begin{align*}
\max \limits_{t>0}E_{\mu,\nu,\lambda}(tu,tv)>\min\{\max \limits_{t>0}f_{\mu-\frac{\lambda^{2}}{\nu}}(tu),\  \max \limits_{t>0}g(tv)\}\geq\frac{s}{N}S^{\frac{N}{2s}}_{s}.
\end{align*}
Thus,
\begin{align}\label{int54}
A_{\mu,\nu,\lambda}\geq\frac{s}{N}S^{\frac{N}{2s}}_{s}.
\end{align}

Combining \eqref{int53} with \eqref{int54}, we obtain $A_{\mu,\nu,\lambda}=\frac{s}{N}S^{\frac{N}{2s}}_{s}.$

{\bf Step 2}
We prove $(i)-(ii)$ in $(2)$.

Let $0<\lambda<\sqrt{\mu\nu}$, we define $$\beta:=\frac{\mu+\nu a_{\mu,\nu}^{2}-2\lambda a_{\mu,\nu}^{2}}{1+a_{\mu,\nu}^{2}},\ \gamma:=\frac{1}{1+a_{\mu,\nu}^{2}},$$
where $a_{\mu,\nu}$ is from \eqref{int48}.

Then,
\begin{align*}
A_{\mu,\nu,\lambda}&\leq\max \limits_{t>0}E_{\mu,\nu,\lambda}(tw_{\beta,\gamma},t(a_{\mu,\nu}w_{\beta,\gamma}))\\
&<(1+a_{\mu,\nu}^{2})\max \limits_{t>0}f_{\beta,\gamma}(tw_{\beta,\gamma})=(1+a_{\mu,\nu}^{2})f_{\beta,\gamma}(w_{\beta,\gamma})\\
&=(1+a_{\mu,\nu}^{2})^{\frac{N}{2s}}\left(\mu+\nu a_{\mu,\nu}^{2}-2\lambda a_{\mu,\nu}\right)^{\frac{p+1}{p-1}-\frac{N}{2s}}\left(\frac{1}{2}-\frac{1}{p+1}\right)C^{\frac{p+1}{p-1}}_{p+1}=:A_{0}.
\end{align*}
By Lemma \ref{int50} we have that $A_{0}\leq\frac{s}{N}S^{\frac{N}{2s}}_{s}$ is equivalent to $$(1+a_{\mu,\nu}^{2})^{\frac{N}{2s}}\left(\mu+\nu a_{\mu,\nu}^{2}-2\lambda a_{\mu,\nu}\right)^{\frac{p+1}{p-1}-\frac{N}{2s}}\leq\mu^{\frac{p+1}{p-1}-\frac{N}{2s}}_{0}.$$
By \eqref{int47} and \eqref{int48}, we can deduce that the above inequality is equivalent to $\lambda>\widetilde{\lambda}_{\mu,\nu}.$ Combining this with \eqref{int49}, for any $\lambda\in[\widetilde{\lambda}_{\mu,\nu},\sqrt{\mu\nu}),$ we have $A_{\mu,\nu,\lambda}<\frac{s}{N}S^{\frac{N}{2s}}_{s}.$ Define $$\lambda_{\mu,\nu}:=\inf\{\lambda<\sqrt{\mu\nu}:A_{\mu,\nu,\tau}<\frac{s}{N}S^{\frac{N}{2s}}_{s}, \forall\ \tau\in[\lambda,\sqrt{\mu\nu})\}.$$ Then, by \eqref{int51}, we know $\lambda_{\mu,\nu}\in[\sqrt{(\mu-\mu_{0})\nu},\widetilde{\lambda}_{\mu,\nu}]$ and for any $\lambda\in(\lambda_{\mu,\nu},\sqrt{\mu\nu}),$ there holds $A_{\mu,\nu,\lambda}<\frac{s}{N}S^{\frac{N}{2s}}_{s}.$ This completes the proof of $(ii)$.

We show that $A_{\mu,\nu,\lambda_{\mu,\nu}}=\frac{s}{N}S^{\frac{N}{2s}}_{s},$ which implies $\lambda_{\mu,\nu}<\widetilde{\lambda}_{\mu,\nu}$ immediately.

By \eqref{int53}, we have $A_{\mu,\nu,\lambda_{\mu,\nu}}\leq\frac{s}{N}S^{\frac{N}{2s}}_{s}.$ By the definition of $\lambda_{\mu,\nu},$ there exists $\lambda_{n}<\lambda_{\mu,\nu},\ n\geq1$ such that $$\lim\limits_{n\rightarrow+\infty}\lambda_{n}=\lambda_{\mu,\nu}, \ \ A_{n}:=A_{\mu,\nu,\lambda_{n}}\geq\frac{s}{N}S^{\frac{N}{2s}}_{s},\ \ \forall\ n\geq1.$$
For any $(u,v)\in \mathcal{H}_{r}\setminus\{(0,0)\},$ there exist $t_{n}>0$ such that
$\max\limits_{t>0}E_{\mu,\nu,\lambda_{n}}(tu,tv)=E_{\mu,\nu,\lambda_{n}}(t_{n}u,t_{n}v).$ Since $\lambda_{n}\rightarrow\lambda_{\mu,\nu},$ we have $t_{n}\rightarrow t_{0}$ as $n\rightarrow+\infty,$ where $t_{0}$ satisfies $\max\limits_{t>0}E_{\mu,\nu,\lambda_{n}}(tu,tv)=E_{\mu,\nu,\lambda_{n}}(t_{0}u,t_{0}v).$ Then, $$\limsup\limits_{n\rightarrow+\infty}A_{n}\leq\limsup\limits_{n\rightarrow+\infty}E_{\mu,\nu,\lambda_{n}}(t_{n}u,t_{n}v)=E_{\mu,\nu,\lambda_{\mu,\nu}}(t_{0}u,t_{0}v).$$
This implies $$\frac{s}{N}S^{\frac{N}{2s}}_{s}\leq\limsup\limits_{n\rightarrow+\infty}A_{n}\leq A_{\mu,\nu,\lambda_{\mu,\nu}}.$$
Thus, $A_{\mu,\nu,\lambda_{\mu,\nu}}=\frac{s}{N}S^{\frac{N}{2s}}_{s}.$

Next, we claim $A_{\mu,\nu,\lambda}$ is non-increasing with respect to $\lambda>0.$

Since
\begin{align}\label{int57}
\max\limits_{t>0}E_{\mu,\nu,\lambda}(tu,tv)\geq\max\limits_{t>0}E_{\mu,\nu,\lambda}(t|u|,t|v|),
\end{align}
then we have
\begin{align}\label{int58}
A_{\mu,\nu,\lambda}=\inf \limits_{\mathcal{H}_{r}\setminus\{(0,0)\}}\max \limits_{t>0}E_{\mu,\nu,\lambda}(t|u|,t|v|).
\end{align}
 Let $\lambda_{1}<\lambda_{2}.$ Then for any $(u,v)\in \mathcal{H}_{r}\setminus\{(0,0)\}$ and $t>0$ we have $$E_{\mu,\nu,\lambda_{1}}(t|u|,t|v|)\geq E_{\mu,\nu,\lambda_{2}}(t|u|,t|v|).$$
 Thus, $A_{\mu,\nu,\lambda_{1}}\geq A_{\mu,\nu,\lambda_{2}}.$ Consequently, $A_{\mu,\nu,\lambda}$ is non-increasing with respect to $\lambda>0.$

 Combining this fact with \eqref{int51}, we see that $(i)$ holds. This completes the proof.
\end{proof}
\section{Proof of Theorem 1.1}\label{sec1}
\begin{proof}[Proof of Theorem \ref{Th1}]
By principle of symmetric criticality (Theorem 1.28 in \cite{MWM}), the solutions for \eqref{int1} in function space ${H_{r}^{s}(\mathbb{R}^{N})}\times {H_{r}^{s}(\mathbb{R}^{N})}$ are also the solutions in function space ${H^{s}(\mathbb{R}^{N})}\times {H^{s}(\mathbb{R}^{N})}$.

We prove Theorem \ref{Th1} by two steps. First, we prove the existence of ground state solutions for system \eqref{int1} in step 1, then we claim there exist a positive ground state solution.

{\bf Step 1}. Prove the existence of ground state solutions for system \eqref{int1}.\\
By \eqref{int4} and the proof of Lemma \ref{Lem5}, there exists a bounded sequence $\{(u_{n},v_{n})\}\subset \mathcal{H}_{r}$, such that
\begin{align*}
E_{\mu,\nu,\lambda}(u_{n},v_{n})\rightarrow \widehat{A}_{\mu,\nu,\lambda}\ \text{and} \ E'_{\mu,\nu,\lambda}(u_{n},v_{n})\rightarrow0 \ \text{as} \ n\rightarrow+\infty.
\end{align*}

Thus, by Sobolev Imbedding Theorem, there exist $(u,v)\in\mathcal{H}_{r}$ such that
 \begin{equation}
\begin{cases}

(u_{n},v_{n})\rightharpoonup(u,v), & \text{weakly in } \ \ \mathcal{H}_{r},\\
(u_{n},v_{n})\rightarrow(u,v),& \text{strongly in} \ \ L^{p}(\mathbb{R}^{N})\times L^{p}(\mathbb{R}^{N}), \text{for} \ 2< p<2^{\ast},\\

(u_{n},v_{n})\rightarrow(u,v), & \text{a.e.}\ \mathbb{R}^{N}.
\end{cases}
\end{equation}
If $(u,v)\not\equiv (0,0)$, then we have
\begin{equation}\label{aaa}
E'_{\mu,\nu,\lambda}(u,v)=0.
\end{equation}
Let $w_{n}=u_{n}-u$ and $\sigma_{n}=v_{n}-v$. We claim
\begin{align}\label{o}
E_{\mu,\nu,\lambda}(w_{n},\sigma_{n})&=E_{\mu,\nu,\lambda}(u_{n},v_{n})-E_{\mu,\nu,\lambda}(u,v)+o(1).
\end{align}
\begin{align}\label{oo}
E_{\mu,\nu,\lambda}'(w_{n},\sigma_{n})=o(1).
\end{align}
By Lemma \ref{k}, there holds
\begin{align}\label{ab}
\|u_{n}\|^{2}_{2}=\|u\|^{2}_{2}+\|w_{n}\|^{2}_{2}+o_{n}(1),
\end{align}
\begin{align}\label{aa}
\|v_{n}\|^{2}_{2}=\|v\|^{2}_{2}+\|\sigma_{n}\|^{2}_{2}+o_{n}(1),
\end{align}
\begin{align}\label{int7}
\|u_{n}\|^{p+1}_{P+1}=\|u\|^{p+1}_{P+1}+\|w_{n}\|^{P+1}_{P+1}+o_{n}(1)
\end{align}
and
\begin{align}\label{int77}
\|v_{n}\|^{2^{\ast}}_{2^{\ast}}=\|v\|^{2^{\ast}}_{2^{\ast}}+\|\sigma_{n}\|^{2^{\ast}}_{2^{\ast}}+o_{n}(1).
\end{align}
Since
\begin{align}\label{int8}
 \|w_{n}\|^{2}_{D_{r}^{s}(\mathbb{R}^{N})}=\|u_{n}\|^{2}_{D_{r}^{s}(\mathbb{R}^{N})}-\|u\|^{2}_{D_{r}^{s}(\mathbb{R}^{N})}+o_{n}(1),\\\nonumber \|\sigma_{n}\|^{2}_{D_{r}^{s}(\mathbb{R}^{N})}=\|v_{n}\|^{2}_{D_{r}^{s}(\mathbb{R}^{N})}-\|v\|^{2}_{D_{r}^{s}(\mathbb{R}^{N})}+o_{n}(1)
 \end{align}
 and
 \begin{align}\label{int88}
 \int_{\mathbb{R}^{N}}w_{n}\sigma_{n}dx=\int_{\mathbb{R}^{N}}u_{n}v_{n}dx-\int_{\mathbb{R}^{N}}uvdx+o(1).
 \end{align}
Combining \eqref{ab}-\eqref{int8} with \eqref{int88} , we obtain identity \eqref{o}.

Since, for any $(\phi,\varphi)\in \mathcal{H}_{r}$, we have
\begin{align}\label{zzzz}
o(\|(\phi,\varphi)\|)&=<E_{\mu,\nu,\lambda}'(u_{n},v_{n}),(\phi,\varphi)>\\\nonumber
&=<E_{\mu,\nu,\lambda}'(w_{n},\sigma_{n}),(\phi,\varphi)>+<E_{\mu,\nu,\lambda}'(u,v),(\phi,\varphi)>\\\nonumber
&-\int_{\mathbb{R}^{N}}\left[|v_{n}|^{2^{\ast}-2}v_{n}-|\sigma_{n}|^{2^{\ast}-2}\sigma_{n}-|v|^{2^{\ast}-2}v\right]\varphi dx\\\nonumber
&-\int_{\mathbb{R}^{N}}\left[|u_{n}|^{p-1}u_{n}-|w_{n}|^{p-1}w_{n}-|u|^{p-1}u\right]\phi dx.
\end{align}
By Vatali's Theorem \cite{WRC}, we have
\begin{align}\label{zz}
 &\int_{\mathbb{R}^{N}}\left[|v_{n}|^{2^{\ast}-2}v_{n}-|\sigma_{n}|^{2^{\ast}-2}\sigma_{n}-|v|^{2^{\ast}-2}v\right]\varphi dx\\\nonumber
 &=\int_{\mathbb{R}^{N}}\int^{1}_{0}\frac{d\left[[tv_{n}+(1-t)\sigma_{n}]^{2^{\ast}-1}-(tv)^{2^{\ast}-1}\right]}{dt}\varphi dx\\\nonumber
 &=\int_{\mathbb{R}^{N}}\int^{1}_{0}(2^{\ast}-1)[[tv_{n}+(1-t)\sigma_{n}]^{2^{\ast}-2}-(tv)^{2^{\ast}-2}]v \varphi dtdx\rightarrow 0\ \text{as}\ n\rightarrow+\infty.
 \end{align}
 Similarly, we have
 \begin{equation}\label{zzz}
 \int_{\mathbb{R}^{N}}\left[|u_{n}|^{p-1}u_{n}-|w_{n}|^{p-1}w_{n}-|v|^{p-1}u\right]\phi dx\rightarrow 0\ \text{as}\ n\rightarrow+\infty.
 \end{equation}
 By \eqref{zzzz}, \eqref{zz}, \eqref{zzz}, we obtain identity \eqref{oo}.

 Thus, by \eqref{aaa},\eqref{o}, \eqref{oo} and Lemma \ref{Lem5}, we deduce
 \begin{align*}
 o(1)+\widehat{A}_{\mu,\nu,\lambda}&=E_{\mu,\nu,\lambda}(u_{n},v_{n})=E_{\mu,\nu,\lambda}(w_{n},\sigma_{n})+E_{\mu,\nu,\lambda}(u,v)+o(1)\\
 &\geq \widehat{A}_{\mu,\nu,\lambda}+E_{\mu,\nu,\lambda}(w_{n},\sigma_{n})+o(1)
 \end{align*}
 which gives that
 \begin{equation}\label{cc}
 E_{\mu,\nu,\lambda}(w_{n},\sigma_{n})\leq o(1).
 \end{equation}
 On the other hand, by \eqref{oo}, we have
 \begin{equation}\label{ccc}
 ||(w_{n},\sigma_{n})||^{2}_{\mathcal{D}_{r}}+\int_{\mathbb{R}^{N}}(\mu w_{n}^{2}+\nu \sigma_{n}^{2})dx-
\int_{\mathbb{R}^{N}}|\sigma_{n}|^{2^{\ast}}dx+2\lambda \int_{\mathbb{R}^{N}}w_{n}\sigma_{n} dx=o(1).
\end{equation}
Since $\lambda<\sqrt{\mu\nu }$, we can take a small $\tau>0$ such that $\lambda^2=(\mu-\tau)(\nu-\tau)$, by Sobolev imbedding, \eqref{cc} and \eqref{ccc}, we find
\begin{align*}
o(1)&\geq\frac{s}{N}\left[\|(w_{n},\sigma_{n})\|^{2}_{\mathcal{D}_{r}}+\int_{\mathbb{R}^{N}}(\mu w^{2}_{n}+\nu \sigma^{2}_{n} )dx-2\lambda\int_{\mathbb{R}^{N}}w_{n}\sigma_{n}dx\right]\\
&\geq\frac{s}{N}\left[||(w_{n},\sigma_{n})||^{2}_{\mathcal{D}_{r}}+\tau\int_{\mathbb{R}^N}(w_{n}^2+\sigma_{n}^2)\right]\\
&\geq\frac{s}{N}\min(1,\tau)||(w_{n},\sigma_{n})||^{2}_{\mathcal{H}_{r}},
\end{align*}
which implies that $$(u_{n},v_{n})\rightarrow(u,v)\ \text{strongly in }\ \mathcal{H}_{r}.$$
As a result, $(u,v)$ is a critical point of $E_{\mu,\nu,\lambda}$ and satisfies $$E_{\mu,\nu,\lambda}(u,v)=\widehat{A}_{\mu,\nu,\lambda}\ \text{and}\ \ E'_{\mu,\nu,\lambda}(u,v)=0.$$
By Lemma \ref{Lem5}, we obtain $$\widehat{A}_{\mu,\nu,\lambda}= \overline{A}_{\mu,\nu,\lambda}=A_{\mu,\nu,\lambda}.$$
 If $(u,v)\equiv (0,0)$, by \eqref{ccc} and
\begin{align*}
E_{\mu,\nu,\lambda}(u_{n},v_{n})=&\frac{1}{2}\|(w_{n},\sigma_{n})\|^{2}_{\mathcal{D}_{r}}+\frac{1}{2}\int_{\mathbb{R}^{N}}(\mu w^{2}_{n}+\nu \sigma^{2}_{n} )dx\\\nonumber
&-\lambda\int_{\mathbb{R}^{N}}w_{n}\sigma_{n}dx-\frac{1}{2^{\ast}}\int_{\mathbb{R}^{N}}|\sigma_{n}|^{2^{\ast}}dx+o_{n}(1).
\end{align*}
We have
\begin{align}\label{int11}
E_{\mu,\nu,\lambda}(u_{n},v_{n})=&\frac{s}{N}\|(w_{n},\sigma_{n})\|^{2}_{\mathcal{D}_{r}}\\\nonumber
&+\frac{s}{N}\left[\int_{\mathbb{R}^{N}}(\mu w^{2}_{n}+\nu \sigma^{2}_{n} )dx-2\lambda\int_{\mathbb{R}^{N}}w_{n}\sigma_{n}dx\right]+o_{n}(1)\\\nonumber
&\geq \frac{s}{N}\|(w_{n},\sigma_{n})\|^{2}_{\mathcal{D}_{r}}+o_{n}(1).
\end{align}
 Assuming by contradiction, we can assume without loss of generality that $$\lim_{n\rightarrow+\infty}\|(w_{n},\sigma_{n})\|^{2}_{\mathcal{D}_{r}}=l>0.$$
By \eqref{ccc} and Cauchy-Schwarz inequality, we have
\begin{align}\label{abc}
\|(w_{n},\sigma_{n})\|^{2}_{\mathcal{D}_{r}}\leq \int_{\mathbb{R}^{N}}|\sigma_{n}|^{2^{\ast}}dx+o_{n}(1).
\end{align}
 By Sobolev imbedding $D_{r}^{s}(\mathbb{R}^{N})\hookrightarrow L^{2^{\ast}}(\mathbb{R}^{N})$, we have
 \begin{align}\label{abcd}
 \|\sigma_{n}\|^{2}_{D_{r}^{s}(\mathbb{R}^{N})}\geq S_{s}\left(\int_{\mathbb{R}^{N}}|\sigma_{n}|^{2^{\ast}}dx\right)^{\frac{2}{2^{\ast}}}.
 \end{align}
 Combining \eqref{abc} with \eqref{abcd}, we can deduce that
 $$l\geq S^{\frac{N}{2s}}_{s}.$$
 Let $n\rightarrow+\infty$ in \eqref{int11}, we obtain
 $$A_{\mu,\nu,\lambda}\geq\frac{s}{N}S^{\frac{N}{2s}}_{s}.$$
 This contradict with Lemma \ref{Lem7}.

Consequently, $E_{\mu,\nu,\lambda}(u,v)=\widehat{A}_{\mu,\nu,\lambda}=A_{\mu,\nu,\lambda}$ and $E'_{\mu,\nu,\lambda}(u,v)=0$. That is $(u,v)$ is a nontrivial solution of system \eqref{int1}.

{\bf Step 2}. We claim that there exist a positive ground state solution.

Since
 \begin{align*}
 &\int_{\mathbb{R}^{N}}\int_{\mathbb{R}^{N}}\frac{|u(x)-u(y)|^{2}}{|x-y|^{N+2s}}dxdy-\int_{\mathbb{R}^{N}}\int_{\mathbb{R}^{N}}\frac{||u(x)|-|u(y)||^{2}}{|x-y|^{N+2s}}dxdy\\
 &=2\int_{\mathbb{R}^{N}}\int_{\mathbb{R}^{N}}\frac{||u(x)||u(y)|-u(x)u(y)|^{2}}{|x-y|^{N+2s}}dxdy\geq0,
 \end{align*}
 hence, $$\||u|\|_{D_{r}^{s}(\mathbb{R}^{N})}\leq\|u\|_{D_{r}^{s}(\mathbb{R}^{N})}.$$
 Then, for the minimizing sequence $(u_{n},v_{n})\in\mathbb{M}_{r}$, we have
 \begin{align*}
 &\||u_{n}|\|^{2}_{D_{r}^{s}(\mathbb{R}^{N})}+\||v_{n}|\|^{2}_{D_{r}^{s}(\mathbb{R}^{N})}+\int_{\mathbb{R}^{N}}(\mu|u_{n}|^{2}+\nu|v_{n}|^{2})dx\\
 &\leq \|u_{n}\|^{2}_{D_{r}^{s}(\mathbb{R}^{N})}+\|v_{n}\|^{2}_{D_{r}^{s}(\mathbb{R}^{N})}+\int_{\mathbb{R}^{N}}(\mu u^{2}_{n}+\nu v^{2}_{n})dx\\
 &=\int_{\mathbb{R}^{N}}(|u_{n}|^{p+1}+|v_{n}|^{2^{\ast}})dx+\int_{\mathbb{R}^{N}}2\lambda |u_{n}||v_{n}|dx,
 \end{align*}
 this implies that there exists $t_{n}\in (0,1]$ such that $(t_{n}|u_{n}|,t_{n}|v_{n}|)\in\mathbb{M}_{r}.$ Hence, we can choose a minimizing sequence $(\overline{u}_{n},\overline{v}_{n})=(t_{n}|u_{n}|,t_{n}|v_{n}|)$ and the weak limit $(\overline{u},\overline{v})$ is nonnegative.
By Strong maximum principle for fractional Laplacian( see,
Proposition 2.17 in \cite{LS}), we have $\overline{u}$ and $\overline{v}$ are both positive.

Next, we claim if $\mu>\mu_{0}$ and $0<\lambda<\lambda_{\mu,\nu}$, then system \eqref{int1} has no ground state solution.

Assume by contradiction that there exist $\lambda\in(0,\lambda_{\mu,\nu})$ such that system \eqref{int1} has a ground state solution $(u_{\lambda},v_{\lambda})\neq (0,0)$. Then $E_{\mu,\nu,\lambda}(u_{\lambda},v_{\lambda})=A_{\mu,\nu,\lambda}=\frac{s}{N}S^{\frac{N}{2s}}_{s}.$ By \eqref{int57} and \eqref{int58} we may assume that $u_{\lambda}\geq0,v_{\lambda}\geq0,$ by Strong maximum principle for fractional Laplacian( see Proposition 2.17 in \cite{LS}), we have $u_{\lambda}>0,v_{\lambda}>0.$ If we take $\lambda_{1}\in(\lambda,\lambda_{\mu,\nu})$, then by Lemma \ref{Lem7} and \eqref{Lem2}, we have
\begin{align*}
\frac{s}{N}S^{\frac{N}{2s}}_{s}&=A_{\mu,\nu,\lambda_{1}}\leq \max \limits_{t>0}E_{\mu,\nu,\lambda_{1}}(tu_{\lambda},tv_{\lambda})\\
&=E_{\mu,\nu,\lambda_{1}}(t_{\lambda_{1},u_{\lambda},v_{\lambda}}u_{\lambda},t_{\lambda_{1},u_{\lambda},v_{\lambda}}v_{\lambda})\\
&=E_{\mu,\nu,\lambda}(t_{\lambda_{1},u_{\lambda},v_{\lambda}}u_{\lambda},t_{\lambda_{1},u_{\lambda},v_{\lambda}}v_{\lambda})-(\lambda_{1}-\lambda)t^{2}_{\lambda_{1},u_{\lambda},v_{\lambda}}\int_{\mathbb{R}^{N}}u_{\lambda}v_{\lambda}dx\\
&<E_{\mu,\nu,\lambda}(t_{\lambda_{1},u_{\lambda},v_{\lambda}}u_{\lambda},t_{\lambda_{1},u_{\lambda},v_{\lambda}}v_{\lambda})\leq E_{\mu,\nu,\lambda}(u_{\lambda,v_{\lambda}})=\frac{s}{N}S^{\frac{N}{2s}}_{s},
\end{align*}
a contradiction. This completes the proof.
\end{proof}

E-mails: mdzhen@hust.edu.cn; taoismnature@hust.edu.cn; hyxu@hust.edu.cn; yangmeih@hust.edu.cn.


\begin{thebibliography}{99}
\bibitem{AAPH} A. Ambrosetti and  P.H. Rabinowitz, Dual variational methods in critical point theory and applications, {\it J. Funct. Anal.}, 14(1973), 349--381.
\bibitem{GGP} G. Alberti, G. Bouchitt¨¦ and P. Seppecher, Phase transition with the line-tension effect, {\it  Arch. Rational Mech. Anal.}, 144(1998), 1--46.

  \bibitem{AMM}C.O. Alves, D.C. de Morais Filho and M.A.S. Souto, On systems of elliptic equations involving subcritical or critical Sobolev exponents, {\it Nonlinear Anal.}, 42(2000), 771--787.
 \bibitem{BCPS}B. Barrios, E. Colorado, A. de Pablo and U. S\'{a}nchez, On some critical problems for the fractional Laplacian operator, {\it J. Differential Equations}, 252(2012), 6133--6162.
   \bibitem{HBEL}H. Brezis and E. Lieb, A relation between pointwise convergence of functions and functionals, {\it Proc. Amer. Math. Soc.}, 88(1983), 486--490.
 \bibitem{HBPL}H. Berestycki and P. L. Lions, Nonlinear scalar field equations (I): Existence of a ground state, {\it Arch. Rational Mech. Anal.}, 82(1983), 247--375.
\bibitem{CS2}X. Cabr\'{e} and Y. Sire, Nonlinear equations for fractional Laplacians, I: Regularity, maximum principles, and Hamiltonian estimates, {\it Ann. Inst. H. Poincar¨¦ Anal. Non Lin¨¦aire}, 31(2014), 23--53.
\bibitem{CRS}L. Caffarelli, J. Roquejoffre and Y. Sire, Variational problems with free boundaries for the fractional Laplacian, {\it J. Eur. Math. Soc.}, 12(2010), 1151--1179.
\bibitem{CDS1}A. Capella, J. D\'{a}vila, L. Dupaigne and Y. Sire, Regularity of radial extremal solutions for some non-local semilinear equations, {\it Comm. Partial Differential Equations}, 36(2011), 1353--1384.

 \bibitem{XCZ}X. Chang and Z.Q. Wang, Ground state of scalar field equations involving fractional Laplacian
with general nonlinearity, {\it Nonlinearity}, 26(2013), 479--494.

 \bibitem{CZ1}Z. Chen and W. Zou, An optimal constant for the existence of least energy solutions of a coupled Schr\"{o}dinger system, {\it Calc. Var. Partial Differential Equations}, 48(2013), 695--711.
 \bibitem{CZ6}Z. Chen and W. Zou, Ground states for a system of Schr\"{o}dinger equations with critical exponent, {\it J. Funct. Anal.}, 262(2012), 3091--3107.

 \bibitem{CZ2}Z. Chen and W. Zou, Positive least energy solutions and phase separation for coupled Schr\"{o}dinger equations with critical exponent, {\it  Arch. Ration. Mech. Anal.}, 205(2012), 515--551.
 \bibitem{CZ3}Z. Chen and W. Zou, Positive least energy solutions and phase separation for coupled Schr\"{o}dinger equations with critical exponent: higher dimensional case, {\it  Calc. Var. Partial Differential Equations}, 52(2015), 423--467.
 \bibitem{CM}X. Cheng and S. Ma, Existence of three nontrivial solutions for elliptic systems with critical exponents and weights, {\it Nonlinear Anal.}, 69(2008), 3537--3548.
\bibitem{CDS2}E. Colorado, A. de Pablo and U. S\'{a}nchez, Perturbations of a critical fractional equation, {\it Pacific J. Math.}, 271(2014), 65--85.
\bibitem{CT}A. Cotsiolis and N.K. Tavoularis,  Best constants for Sobolev inequalities for higher order fractional derivatives, {\it J. Math. Anal. Appl.}, 295(2004), 225--236.
\bibitem{DPV}E. Di Nezza, G. Palatucci and E. Valdinoci, Hitchhiker¡¯s guide to the fractional Sobolev spaces, {\it Bull. Sci. Math.}, 136(2012), 521--573.
\bibitem{SGE}S. Dipierro, G. Palatucci and E. Valdinoci, Existence and symmetry results for a Schr\"{o}dinger
type problem involving the fractional Laplacian, {\it Le Matematiche}, 68(2013), 201--216.
\bibitem{QGX}Q. Guo and X. He, Least energy solutions for a weakly coupled fractional Schr\"{o}dinger system,
 {\it Nonlinear Anal.}, 132(2016), 141--159.
\bibitem{GLZ}Z. Guo, S. Luo and W. Zou, On critical systems involving fractional Laplacian,
 {\it J. Math. Anal. Appl.}, 446(2017), 681--706.
\bibitem{XSZ}X. He, M. Squassina and W. Zou, The Nehari manifold for fractional systems involving critical nonlinearities, {\it Commun. Pure Appl. Anal.}, 15(2016), 1285--1308.
\bibitem{YM}Y. Hua and M.X. Yu, On the ground state solution for a critical fractional Laplacian equation, {\it Nonlinear Anal.}, 87(2013), 116--125.
\bibitem{DFP}D.F. L\"{u} and S.J. Peng, On the positive vector solutions for nonlinear fractional Laplacian system with linear coupling, {\it Discrete Contin. Dys. Syst.}, 37(2017), 3327--3352.
\bibitem{RLFL}R.L. Frank and E. Lenzmann, Uniqueness of non-linear ground states
for fractional Laplacians in $\mathbb{R}$, {\it Acta Math.}, 210(2013), 261--318.
\bibitem{FLS}R.L. Frank, E. Lenzmann and L. Silvestre, Uniqueness of radial solutions for the fractional Laplacian
, {\it Comm. Pure Appl. Math.}, 69(2016), 1671--1726.
\bibitem{QZY}Q. Li and Z. Yang, Multiple positive solutions for a fractional Laplacian
system with critical nonlinearities, {\it Bull. Malays. Math. Sci. Soc.}, 2(2016), 1--27.
\bibitem{MF}J. Marcos and D. Ferraz, Concentration-compactness principle for nonlocal scalar field equations with critical growth, {\it J. Math. Anal. Appl.}, 449(2017), 1189--1228.
 \bibitem{MMC} A. Mellet, S. Mischler and C. Mouhot, Fractional diffusion limit for collisional kinetic equations, {\it Arch. Ration. Mech. Anal. }, 199(2011), 493--525.
 \bibitem{PWW} S.J. Peng, S. Wei and Q.F. Wang, Multiple positive solutions for linearly coupled nonlinear elliptic systems with critical exponent, {\it J. Differential Equations }, 263(2017), 709--731.
   \bibitem{PPW} S.J. Peng, Y.F. Peng and Z.Q. Wang, On elliptic systems with Sobolev critical growth, {\it Calc. Var. Partial Differential Equations ,} 55(2016),Art. 142, 30 pp.
\bibitem{WRC}W. Rudin, Real and complex analysis, 3rd edition, {\it McGraw-Hill Book Co.}, New York, 1987.
 \bibitem{XR} X.Ros-Oton, Nonlocal elliptic equations in bounded domains: a survey, {\it Publ. Mat.}, 60(2016), 3--26.
 \bibitem{SV2}R. Servadei and E. Valdinoci,  Weak and viscosity solutions of the fractional Laplace equation, {\it Publ. Mat.}, 58(2014), 133--154.
 \bibitem{SVa}R. Servadei and E. Valdinoci, The Brezis-Nirenberg result for the fractional Laplacian, {\it Trans. Amer. Math. Soc.}, 367(2015), 67-102.
\bibitem{SZY}X. Shang, J. Zhang and Y. Yang, Positive solutions of nonhomogeneous fractional Laplacian problem with critical exponent,{\it Commun. Pure Appl. Anal.}, 13(2014), 567--584.
 \bibitem{LS}L. Silvestre, Regularity of the obstacle problem for a fractional power of the Laplace operator, {\it Comm. Pure Appl. Math.}, 60(2007), 67--112.
 \bibitem{ZWH}Z. Wang and H.S. Zhou, Radial sign-changing solution for fractional Schr\"{o}dinger equation, {\it Discrete Contin. Dyn. Syst.}, 36(2016), 449--508.
\bibitem{MWM}M. Willem, Minimax Theorems, {\it Department de Mathematique University Catholique de Louvain.}
\bibitem{ZHX}M.D. Zhen, J.C. He and H.Y. Xu, Critical system involving fractional Laplacian, {\it Commun. Pure Appl. Anal.}, 1(2019), 237--253.
\end{thebibliography}
\end{document}